\documentclass[12pt]{article}
\usepackage{wrapfig}
\usepackage{graphicx}
\usepackage[utf8x]{inputenc}
\usepackage{amssymb}
\usepackage{enumerate}
\usepackage{amsmath}
\usepackage{amsthm}

\newtheorem{theorem}{Theorem}[section]

\newtheorem{lemma}[theorem]{Lemma}
\newtheorem{remark}[theorem]{Remark}

\title{On the Cauchy problem for a model equation for shallow water waves
of moderate amplitude}
\author{Nilay Duruk Mutluba\c{s}\\
University of Vienna, Faculty of Mathematics,\\
Nordbergstrasse 14, 1090 Vienna, Austria\\
nilay.duruk.mutlubas@univie.ac.at}
\begin{document}

\maketitle
 \begin{abstract}
  \noindent We prove the local well-posedness for a nonlinear equation modeling the evolution
  of the free surface for waves of moderate amplitude in the shallow water regime.
 \end{abstract}

 

 \section{Introduction}

Water waves and their model equations have drawn attention all the time due to their familiar nature.
Nevertheless, most of the studies were restricted to linear models. Since linearization failed to explain some important
aspects, several nonlinear models have been proposed, explaining nonlinear behaviours such as breaking waves and solitary waves.
A typical example is the Korteweg-de-Vries (KdV) equation \cite{KdV} which models shallow water wave propagation:
\begin{equation*}
u_t+u_x+uu_x+u_{xxx}=0.
\end{equation*}
Physically, the steeping effect of the nonlinearity, represented by $uu_x$, and the smoothing
effect of dispersion, represented by $u_{xxx}$, are in balance with each other. This leads to a remarkable property
of the solitary wave solutions of KdV: they are solitons, recovering their shape and speed after collisions with other waves of this type.

Since KdV does not model breaking waves, several model equations were proposed to capture this phenomenon, in the
sense that they should have classical solutions such that the wave profile remains bounded but its slope becomes
unbounded (see the discussions in \cite{Wh} and \cite{CE}). The proposed equations were mostly mathematical modifications of KdV, with a
quite remote connection to the modeling of water waves. In recent years, many studies were devoted to the
Camassa-Holm (CH) equation \cite{CH}
\begin{equation*}
u_t+u_x+3uu_x-u_{xxt}=2u_xu_{xx}+uu_{xxx}.
\end{equation*}
This equation has a rich structure, being an integrable infinite-dimensional Hamiltonian system (see the discussion in 
\cite{Con11}), and it has bounded classical solutions whose slope becomes unbounded in finite time. In \cite{Jn} and in
\cite{ConLn09} it was shown that both CH arises as a model describing the evolution of the horizontal fluid velocity
at a certain depth within the regime of shallow water waves of moderate amplitude. In terms of the two fundamental
parameters $\mu$ (shallowness parameter) and $\varepsilon$ (amplitude parameter), the shallow water regime
of waves of small amplitude (proper to KdV) is characterized by $\mu \ll 1$ and $\varepsilon=O(\mu)$, while
the regime of shallow water waves of moderate amplitude (proper to CH) corresponds to $\mu \ll 1$ and
$\varepsilon=O(\sqrt{\mu})$; see \cite{AL} and \cite{ConLn09}. Since quantities of order $O(\sqrt{\mu})$ are
also of order $O(\mu)$ for $\mu \ll1$, the regime of moderate amplitude captures a wider range of wave profiles.
In particular, within this regime one expects to obtain equations that model surface water wave profiles that
develop singularities in finite time in the form of breaking waves. The model equation for the evolution of
the surface elevation $\eta(x,t)$ is
\begin{equation}
\label{modamp}
\eta_t+\eta_x+\frac{3}{2}\varepsilon\eta\eta_x-\frac{3}{8}\varepsilon^{2}\eta^{2}\eta_x+\frac{3}{16}\varepsilon^{3}\eta^{3}\eta_x+\mu(\alpha\eta_{xxx}+\beta\eta_{xxt})
=\varepsilon\mu(\gamma\eta\eta_{xxx}+\delta\eta_x\eta_{xx}).
\end{equation}
Here $\alpha$, $\gamma$, $\delta$ and $\beta<0$ are parameters.
The local well-posedness of (\ref{modamp}) for any initial data $\eta_0\in H^{s+1}(\mathbb{R})$ with $s>\frac{3}{2}$ was
proved in \cite{ConLn09}, and the existence of solitary waves was recently obtained in \cite{G}. Note that, unlike KdV or CH,
the equation (\ref{modamp}) does not have a bi-Hamiltonian integrable structure (see \cite{Con11}). Nevertheless, the equation
possesses solitary wave profiles that resemble those of CH, analyzed in \cite{CS}, and present similarities with the shape of the solitary waves
for the governing equations for water waves discussed in \cite{CE2, CEH}, as proved in \cite{G}.

In this paper, we will prove the local well-posedness of the equation (\ref{modamp}) for initial data in
$H^s({\mathbb R})$ with $s>3/2$ by using Kato's semigroup approach for quasilinear equations as it was used in \cite{Rod01}
where the well-posedness of CH equation was investigated. The motivation for proving well-posedness for
less regular initial data stems from the fact that by enlarging the class of initial data, one facilitates the task of finding initial
profiles that develop singularities in finite time in the form of breaking waves. This important aspect of (\ref{modamp}) will
be addressed in a subsequent publication.

Since we will use Kato's theory, in Section 2 we give a brief summary of this approach. In Section 3, we present
the well-posedness results for (\ref{modamp}). Throughout this paper, subscripts denote partial derivatives;
$\partial_x=\partial/\partial x$; $||.||_{X}$ denotes norm in the Banach space $X$,
in particular, $||.||$ is $L^2$ norm;  $H^{s}$ is the classical Sobolev space with norm
$||.||_{H^s}=||.||_{s}$ and $<,>_s$ is its inner product; $\Lambda^s=(1-\partial_x^2)^{s/2}$, $s\in\mathbb{R}$; $[A,B]$
denotes the commutator of the linear operators $A$ and $B$.

 \section{Preliminaries}

 In this section, we state Kato's theorem in a suitable form for our purpose.

Consider the abstract quasi-linear evolution equation in the Hilbert space $X$:
  \begin{equation}
  \label{qlee}
     u_t =A(u)u + f(u), ~~~~t\geq 0,~~~~~u(0)=u_0.
  \end{equation}
Let $Y$ be a second Hilbert space such that $Y$ is continuously and densely injected into $X$ and let $S:Y\rightarrow X$ be a topological isomorphism. Assume that

\begin{itemize}
\item[(A1)] Given $C>0$, for every $y\in Y$ with $||y||_Y\leq C$, $A(y)$ is quasi-m-accretive on $X$, i.e.
$A(y)$ is the generator of a $C_0$ semigroup $\{T(t)\}_{t\geq 0}$ in $X$ satisfying $||T(t)||\leq Me^{\omega t}$ with $M=1$.
\item[(A2)] For every $y\in Y$, $A(y)$ is bounded linear operator from $Y$ to $X$ and
\begin{displaymath}
||(A(y)-A(z))\omega||_X\leq c_1||y-z||_X||\omega||_Y,~~~~y,z,\omega\in Y.
\end{displaymath}
\item[(A3)] For every $C>0$, there is a constant $c_2(C)$ such that $SA(y)S^{-1}=A(y)+B(y)$, for some bounded
linear operator $B(y)$ on $X$ satisfying
 \begin{displaymath}
||(B(y)-B(z))\omega||_X\leq c_2(C)||\omega||_X,~~~~\omega\in X.
\end{displaymath}
\item[(A4)] The function $f$ is bounded in $Y$ and Lipschitz in $X$ and $Y$, i.e.
\begin{equation*}
||f(y)||_Y\leq M
\end{equation*}
for some constant $M>0$. Moreover,
\begin{equation*}
||f(y)-f(z)||_X\leq c_3||y-z||_X,~~~~~\forall y,z\in X
\end{equation*}
and
\begin{equation}
||f(y)-f(z)||_Y\leq c_4||y-z||_Y,~~~~~\forall y,z\in Y.\label{lip}
\end{equation}
\end{itemize}
Here $c_1$, $c_2$, $c_3$ and $c_4$ are non-negative constants.

\begin{theorem} \cite{KatoI, KatoII} Assume (A1), (A2), (A3), (A4) hold.
Given $u_0\in Y$, there is a maximal $T>0$, depending on $u_0$, and a unique solution $u$ to (\ref{qlee}) such that
\begin{displaymath}
u=(u_0,.)\in C([0,T),Y)\cap C^1([0,T),X).
\end{displaymath}
Moreover, the map $u_0\rightarrow u(u_0,.) $ is continuous from $Y$ to \mbox{$C([0,T),Y)\cap C^1([0,T),X)$}.
\end{theorem}

 \section{Local Theory}

Consider the initial value problem for the general class of equations
  \begin{align}
    \label{MAE}
  & u_t + u_x + \frac{3}{2}\varepsilon u u_x +\varepsilon^{2}\iota u^2u_x + \varepsilon^{3}\kappa u^3u_x+\mu(\alpha u_{xxx}+\beta u_{xxt}) \nonumber\\
  & = \varepsilon\mu (\gamma uu_{xxx} + \delta u_xu_{xx} ),~~~~x\in \mathbb{R}, t>0\\
  & u(x,0)= u_0(x)~~~~x\in \mathbb{R} \label{IV}
  \end{align}
  where $\iota,\kappa\in \mathbb{R}$, $\beta<0$, and $u(x,t)$ denotes the free surface. 
\noindent
Our purpose is to prove local existence of the solution for the Cauchy problem (\ref{MAE})-(\ref{IV}). 
Rewriting (\ref{MAE}) in the form of quasi-linear evolution equation (\ref{qlee}), the initial value problem which is equivalent to (\ref{MAE})-(\ref{IV}) will be as follows:
  \begin{align}
   \label{SG}
    & u_t =- (\frac{\alpha}{\beta}\partial_x-\frac{\varepsilon\gamma}{\beta}u\partial_x )u+f(u) \\
    & u(x,0)= u_0(x) \label{IV2}
  \end{align}
  where
  \begin{align}
f(u)= -(1+\mu\beta\partial_x^2)^{-1}\partial_x[(1-\frac{\alpha}{\beta})u+(\frac{3\varepsilon-2}{4}-\frac{1}{2\mu\beta})u^{2}\\
  +\frac{\varepsilon^{2}\iota}{3}u^{3}+\frac{\varepsilon^{3}\kappa}{4}u^{4}+\frac{3\varepsilon\mu\gamma-\varepsilon\mu\delta-\mu\beta}{2}u_x^2]. \label{f}
  \end{align}

  \begin{theorem}
   Let $u_0\in H^s(\mathbb{R})$, $s>\frac{3}{2}$ be given. Then there exists $T>0$, depending on $u_0$, such that there is a unique solution $u$ to
   (\ref{SG})-(\ref{IV2}) satisfying
   \begin{equation*}
    u=u(u_0,.)\in C([0,T),H^s(\mathbb{R}))\cap C^1([0,T),L^{2}(\mathbb{R})).
   \end{equation*}
Moreover, the map $u_0\in H^s(\mathbb{R})\rightarrow u(u_0,.) $ is continuous from $H^s(\mathbb{R})$ to \\
\mbox{$C([0,T),H^s(\mathbb{R}))\cap C^1([0,T),L^2(\mathbb{R}))$}.
  \end{theorem}

We will apply Kato's theorem with $X=L^{2}(\mathbb{R})$, $Y=H^{s}(\mathbb{R})$ with $s>3/2$, $S=\Lambda^{s}$ with 
$\Lambda=(1-\partial_x^2)^{1/2}$. For convenience, we neglect the exact value of the various constants, since the only significant 
feature is that $\beta<0$. The following lemmas are needed to prove Theorem 3.1.

\begin{lemma}
The operator $A(u)=u\partial_x+\partial_x$ with domain 
$\mathcal{D}(A)=\{\omega\in L^{2}(\mathbb{R}):(1+u)\omega\in H^1(\mathbb{R})\}\subset L^{2}(\mathbb{R})$ 
is quasi-m-accretive if $u\in H^s$, $s>3/2$.
\end{lemma}

\begin{proof}
We have an operator of the form $(u(x)+1)\partial_x$ where $u\in C^1(\mathbb{R})\cap L^{\infty}(\mathbb{R})$ and
 $u'\in L^{\infty}(\mathbb{R})$. We refer to Appendix 6.3.1 (E1) in \cite{Con11} to conclude that 
 $A(u)$ is a quasi-m-accretive operator. 
\end{proof}

\begin{lemma}
For every $\omega\in H^s$ with $s>3/2$, $A(u)$ is bounded linear operator from $H^s(\mathbb{R})$ to $L^{2}(\mathbb{R})$ and
\begin{displaymath}
||(A(u)-A(v))\omega||\leq c_1||u-v||||\omega||_{H^{s}}.
\end{displaymath}
\end{lemma}

\begin{proof}
Given $\omega\in H^s(\mathbb{R})$ with $s>3/2$,
\begin{align*}
||(u\partial_x+\partial_x)\omega||& \leq ||u\partial_x\omega||+||\partial_x\omega|| \leq ||u||||\partial_x\omega||_{s-1}+||\partial_x\omega|| \\&\leq ||u||||\partial_x\omega||_{s-1}+||\partial_x\omega||_{s-1}\leq c_1 ||u||||\omega||_{s},
\end{align*}
in view of Lemma 5.1. Assumption (A2) follows from replacing $u$ by $u-v$ in the inequality.
\end{proof}

\begin{lemma}
The operator
$$B(u)=\Lambda^{s}(u\partial_x+\partial_x)\Lambda^{-s}-(u\partial_x+\partial_x)=[\Lambda^{s},u\partial_x+\partial_x]\Lambda^{-s}$$
is bounded in $L^{2}(\mathbb{R})$ for $u\in H^s(\mathbb{R})$ with $s>3/2$.
\end{lemma}

\begin{proof}
Note that
$$
\Lambda^{s}(u\partial_x+\partial_x)\Lambda^{-s}-(u\partial_x+\partial_x)
=\Lambda^{s}u\partial_x\Lambda^{-s}+\Lambda^{s}\partial_x\Lambda^{-s}-(u\partial_x+\partial_x)=[\Lambda^{s},u\partial_x]\Lambda^{-s}
$$
since $\partial_x$ and $\Lambda$ commute. Moreover, we have 
$[\Lambda^{s},u\partial_x]\Lambda^{-s}=[\Lambda^{s},u]\Lambda^{-s}\partial_x$, so that
\begin{eqnarray}\label{map}
||B(u)\omega|| &=& ||[\Lambda^{s},u]\Lambda^{-s}\partial_x\omega||
= ||[\Lambda^{s},u]\Lambda^{1-s}\Lambda^{-1}\partial_x\omega|| \nonumber\\
&\leq& ||u||_s||\Lambda^{-1}\partial_x\omega||=||u||_s||\omega||.
\end{eqnarray}
in view of Lemma 5.2 for $\tilde{s}=0$ and $\tilde {t}=s-1$.
\end{proof}

\begin{remark}
If we replace $u$ with $u-v$ in (\ref{map}), it can be easily observed that
\begin{displaymath}
||B(u)-B(v)\omega||\leq ||\omega||||u-v||_s,
\end{displaymath}
thus proving (A3).
\end{remark}

\begin{lemma}
Let $f(u)$ be given by (\ref{f}). Then:\\
(i) $||f(u)||_s\leq M$ for some constant $M$ depending on $||u||_s$. \\
(ii)$||f(u)-f(v)||\leq c_3||u-v||$.\\
(iii)$||f(u)-f(v)||_s\leq c_4||u-v||_s$, $s>3/2$.
\end{lemma}

\begin{proof}
Observe that
\begin{displaymath}
f(u)= -(1+\mu\beta\partial_x^2)^{-1}\partial_x g(u)
\end{displaymath}
and
\begin{align*}
||f(u)-f(v)|| &= ||(1+\mu\beta \partial_x^2)^{-1}\partial_x(g(u)-g(v))||\\
              &= ||(1-\mu\beta\xi^2)^{-1}(i\xi)(\mathcal{F}g(u)-\mathcal{F}g(v))||
\end{align*}
by using Fourier transform representation. Recall that $\beta<0$. Since $\mu$ is so small, we also assume $|\mu\beta|<1$.
It gives
\begin{align*}
(1-\mu\beta\xi^2)^{-1}\xi&=(1+|\mu\beta|\xi^2)^{-1}(\xi^2)^{1/2}\\
                         &\leq (1+|\mu\beta|\xi^2)^{-1}(1+\xi^2)^{1/2}\\
                         &=|\mu\beta|^{-1}(\frac{1}{|\mu\beta|}+\xi^2)^{-1}(1+\xi^2)^{1/2}\\
                         &\leq |\mu\beta|^{-1}(1+\xi^2)^{-1}(1+\xi^2)^{1/2}\\
                         &=|\mu\beta|^{-1}(1+\xi^2)^{-1/2}.
\end{align*}
Thus,
\begin{align}
||f(u)-f(v)|| \leq& |\mu\beta|^{-1}||(1+\xi^2)^{-1/2}(\mathcal{F}g(u)-\mathcal{F}g(v))||\nonumber\\
              =&|\mu\beta|^{-1}||(1-\partial_x^2)^{-1/2}(g(u)-g(v))||\nonumber\\
              =&|\mu\beta|^{-1}||\Lambda^{-1}(g(u)-g(v))||\nonumber\\
              \leq& |\mu\beta|^{-1}(||\Lambda^{-1}(u-v)||+||\Lambda^{-1}(u^2-v^2)||+||\Lambda^{-1}(u^3-v^3)||\nonumber\\
              &+||\Lambda^{-1}(u^4-v^4)||+||\Lambda^{-1}(u_x^2-v_x^2)||)\label{ii}.
\end{align}
Since,
\begin{align*}
|\mu\beta|(\ref{ii})&=||\Lambda^{-1}(u-v)||+||\Lambda^{-1}(u-v)(u+v)||+||\Lambda^{-1}(u-v)(u^2+uv+v^2)||\\
          &+||\Lambda^{-1}(u-v)(u+v)(u^2+v^2)||+||\Lambda^{-1}\partial_x(u-v)\partial_x(u+v)||,
\end{align*}
using the imbedding property of Sobolev spaces $H^s(\mathbb{R})$, i.e. if $s_1\leq s_2$, then $||.||_{s_1}\leq ||.||_{s_2}$; Cauchy-Schwartz inequality; and Sobolev embedding theorem,
\begin{align*}
||f(u)-f(v)||\leq& |\mu\beta|^{-1}(||(u-v)||+||(u-v)(u+v)||\\
	      &+||(u-v)(u^2+uv+v^2)||+||(u-v)(u+v)(u^2+v^2)||\\
	      &+||\partial_x(u-v)\partial_x(u+v)||_{-1})\\
            \leq& |\mu\beta|^{-1}(||(u-v)||+||(u+v)||_{L^{\infty}}||(u-v)||\\
            &+||(u^2+uv+v^2)||_{L^{\infty}}||(u-v)||+||(u+v)(u^2+v^2)||_{L^{\infty}}||(u-v)||\\
	    &+||\partial_x(u+v)||_{L^{\infty}}||\partial_x(u-v)||_{-1})\\
            \leq& |\mu\beta|^{-1}(||(u-v)||+||(u+v)||_{s}||(u-v)||\\
            &+||(u^2+uv+v^2)||_{s}||(u-v)||+||(u+v)(u^2+v^2)||_{s}||(u-v)||\\
            &+||(u+v)||_{s}||(u-v)||)\\
            \leq& c_3 ||u-v||
\end{align*}
where $c_3$ is a constant depending on $\mu$, $\beta$, $||u||_s$ and $||v||_s$. This proves (ii).

\noindent
Now, we will prove (iii):
\begin{align*}
||f(u)-f(v)||_{s} =& ||(1+\mu\beta \partial_x^2)^{-1}\partial_x(g(u)-g(v))||_s\\
                  \leq& ||(u-v)||_{s-1}+||(u-v)(u+v)||_{s-1}\\
                  &+||(u-v)(u^2+uv+v^2)||_{s-1}+||(u-v)(u+v)(u^2+v^2)||_{s-1}\\
                  &+||\partial_x(u-v)\partial_x(u+v)||_{s-1}\\
                  \leq& ||(u-v)||_{s}+||(u-v)||_s||(u+v)||_{s}\\
                  &+||(u-v)||_s||(u^2+uv+v^2)||_{s}+||(u-v)||_s||(u+v)(u^2+v^2)||_{s}\\
                  &+||\partial_x(u-v)||_{s-1}||\partial_x(u+v)||_{s-1}\\
                  \leq& c_4 ||u-v||_s
\end{align*}
where $c_4$ is also a constant depending on $||u||_s$ and $||v||_s$. Note that (i) can be obtained from (iii) by choosing $v=0$.
\end{proof}

\noindent
\title{\textbf{Proof of Theorem 3.1}}
The proof follows from the lemmas above since the assumptions needed for Kato's semigroup approach are satisfied.

\section{Acknowledgement}
The support of the The Scientific and Technological Research Council of Turkey (TUBITAK) is gratefully acknowledged. 
The author also thanks Prof.Dr.Adrian Constantin for his helpful comments and suggestions.

\section{Appendix}
\begin{lemma}

Let $s$, $t$ be real numbers such that $-s<t\leq s$. Then
\begin{displaymath}
||f.g||_t\leq c||f||_s||g||_t~~~~if ~~s>1/2.
\end{displaymath}
\end{lemma}

\begin{lemma}\cite{KatoII}
Let $f\in H^s$, $s>3/2$ and $M_f$ be the multiplication operator by $f$. Then, for $|\tilde{t}|,|\tilde{s}|\leq s-1$,
\begin{displaymath}
||\Lambda^{-\tilde{s}}[\Lambda^{\tilde{s}+\tilde{t}+1},M_f]\Lambda^{-\tilde{t}}\omega||\leq c||f||_s||\omega||.
\end{displaymath}
\end{lemma}

\end{document}